\documentclass[12pt]{article}
\usepackage[latin1]{inputenc}
\usepackage[british]{babel}
\usepackage{lmodern}
\usepackage[T1]{fontenc}
\usepackage[paper=a4paper, left=28mm, right=25mm, top=29mm, bottom=25mm]{geometry}
\usepackage{latexsym,amsfonts,amsmath}
\usepackage{graphicx}
\usepackage{color}
\usepackage{enumitem}
\usepackage{amssymb,mathrsfs}
\usepackage{verbatim}

\usepackage{algorithm,algorithmic}

\newtheorem{theorem}{Theorem}
\newtheorem{lemma}{Lemma}

\newtheorem{proposition}{Proposition}

\newtheorem{definition}{Definition}
\newtheorem{remark}{Remark}
\newenvironment{proof}{\begin{trivlist} \item[\hskip\labelsep{\it Proof.}]}{$\hfill\Box$\end{trivlist}}

\newcommand{\bsx}{\boldsymbol{x}}

\newcommand{\RR}{\mathbb{R}}
\newcommand{\QQ}{\mathbb{Q}}

\newcommand{\NN}{\mathbb{N}}

\newcommand{\cH}{\mathcal{H}}
\newcommand{\cP}{\mathcal{P}}

\newcommand{\cE}{\mathcal{E}}
\newcommand{\cU}{\mathcal{U}}

\allowdisplaybreaks

\title{Construction algorithms for plane nets in base $b$}
\author{Gunther Leobacher, Friedrich Pillichshammer\thanks{The first two authors are supported by the
Austrian Science Fund (FWF): Projects F5508-N26 (Leobacher)  and
F5509-N26 (Pillichshammer), respectively, which are part of the
Special Research Program ``Quasi-Monte Carlo Methods: Theory and
Applications''.} \ and Thomas Schell\thanks{The third author is supported by the Palfinger AG.}}
\date{}

\begin{document}

\maketitle

\begin{abstract}
The class of $(0,m,s)$-nets in base $b$ has been introduced by Niederreiter as examples of point sets in the $s$-dimensional unit cube with excellent uniform distribution properties. In particular such nets have been proved to have very low discrepancy. This property is essential for the use of nets in quasi-Monte Carlo rules for numerical integration. In this short note we propose two algorithms for the construction of plane $(0,m,2)$-nets in base~$b$.
\end{abstract}

\centerline{\begin{minipage}[hc]{130mm}{
}
\end{minipage}}

\allowdisplaybreaks

\section{Introduction}

In many applications, notably numerical integration based on quasi-Monte Carlo rules, one requires very uniformly distributed point sets in the unit-cube. The quality of the distribution is usually measured by the discrepancy. For a point set $\cP=\{\bsx_1,\ldots,\bsx_{N}\}$ in $[0,1)^s$ and subintervals $J \subseteq [0,1)^s$ the local discrepancy is defined as $$\Delta(J;\cP):=\frac{\#\{n \in \{1,\ldots,N\}\ : \ \bsx_n \in J\}}{N}-{\rm Vol}(J).$$ The discrepancy of $\cP$ is then defined as $$D_N(\cP)=\sup_J |\Delta(J;\cP)|,$$ where the supremum is extended over all intervals $J \subseteq [0,1)^s$. Point sets with discrepancy of order $D_N \ll (\log N)^{s-1}/N$ are called ``low-discrepancy point sets''.

One class of point sets with excellent distribution properties are $(t,m,s)$-nets in base $b$ as introduced by Niederreiter~\cite{nie87} (see also \cite{niesiam}):

Let $s,b \in \NN$, $b \ge 2$. An elementary $b$-adic interval (or box) is an interval of the form $$\prod_{j=1}^s \left[\frac{a_j}{b^{d_j}},\frac{a_j+1}{b^{d_j}}\right)$$ with $d_1,\ldots,d_s \in \NN_0$ and $a_j \in \{0,1,\ldots,b^{d_j}-1\}$.

\begin{definition}[Niederreiter]\rm
Let $b,s,m,t$ be integers such that $b \ge 2$, $s \ge 1$, $m\ge 0$ and $0 \le t \le m$. A $b^m$-element point set $\cP$ in the $s$-dimensional unit cube is called a $(t,m,s)$-net in base $b$ if every elementary $b$-adic interval of volume $b^{t-m}$ contains exactly $b^t$ points of $\cP$.    
\end{definition}

Thus a $(t,m,s)$-net in base $b$ is a $b^m$-element point set $\cP$ in the unit-cube for which the local discrepancy satisfies $\Delta(E;\cP)=0$ for all elementary $b$-adic intervals $E$ of volume $b^{t-m}$. The smaller $t$ is (in the optimal case it is 0), the more demanding is this condition. This property for the local discrepancy is transferred in some sense also to arbitrary sub-intervals of $[0,1)^s$ which still have very low local discrepancy, although it cannot be zero in general. This is reflected in the  discrepancy bounds which are of the form $$D_{b^m}(\cP) \ll_{s,b} b^t \frac{m^{s-1}}{b^m}\ll b^t \frac{(\log N)^{s-1}}{N}\ \ \ \mbox{ if $N=b^m$},$$ and if $\cP$ is a $(t,m,s)$-net in base $b$ (see one of \cite{DP10,LP14,nie87,niesiam}).\\

Motivated by the discrepancy bounds the parameter $t$ in the definition of nets is called the ``quality parameter'' and, as already indicated, a quality parameter which is as small as possible would be appreciated. Unfortunately, the optimal value $t=0$ is not achievable for all possible choices of parameter pairs $(b,s)$. It is well known that a $(0,m,s)$-net in base $b$ can only exist if $s \le b+1$. This follows easily from the following proposition: 

\begin{proposition} \label{th:nonet}
There is no $(0,2,b+2)$-net in base $b$.
\end{proposition}

This result is very well known and there are several proofs available in the literature (see, e.g., \cite{DP10,LP14,nie87,niesiam}). Nevertheless we give here a short and new proof which is based on arguments from Graph Theory and which might bring some new aspects into the theory of $(t,m,s)$-nets. 

\begin{proof}
Suppose to the contrary that there is a $(0,2,b+2)$-net in base $b$. Then its projection to the hyperplane 
orthogonal to the $(b+2)$nd coordinate axis forms a $(0,2,b+1)$-net.
To every point in this net we may assign one of the colors 
$\{0,\ldots,b-1\}$, namely the first digit of the $(b+2)$nd coordinate
of the point in the original net.

Now we construct a graph $G=(V,E)$ by setting $V$ equal to the points
in the projection (hence $|V|=b^2$) and letting $\{v_1,v_2\}\in E$ iff $v_1$ and $v_2$ lie
in the same $b+1$-dimensional elementary box of volume $b^{-1}$. According to our construction
two adjacent vertices have different colors.

Every point in $V$ is contained in $b+1$ such elementary intervals
and every such interval contains exactly $b$ points. Thus
every vertex of $G$ has degree $r=(b+1)(b-1)=b^2-1=|V|-1$. Thus $G$ is isomorphic to the complete graph with $b^2$ vertices, for which 
we have found a coloring by $b$ colors. But this is of course impossible.
\end{proof}

Usually, $(t,m,s)$-nets are constructed with the digital method which is based on $m \times m$-matrices over a finite commutative ring with $b$ elements (one matrix per coordinate).
If $b$ is a prime power, then there always exists a $(0,m,b+1)$-net in base $b$. More information about $(t,m,s)$-nets can be found in \cite{DP10,LP14,nie87,niesiam}.

In this short note we will mainly be concerned with the two-dimensional case. Here one example of a $(0,m,2)$-net in base $b$ is the well known Hammersley net in base $b$ 
\begin{equation}\label{Ham}
\cH_{m,b}=\Big\{\Big(\frac{t_m}{b}+\frac{t_{m-1}}{b^2}+\cdots + \frac{t_1}{b^m},\frac{t_{1}}{b}+\frac{t_{2}}{b^2}  +\cdots + \frac{t_m}{b^m} \Big) \ : \ 
                       t_1,\ldots,t_m \in \{0,1,\ldots,b-1\}   \Big\}.
\end{equation}

Combining results by Dick and Kritzer~\cite{DiKri} and by De Clerck~\cite{DeCl} we obtain:  for every $(0,m,2)$-net $\cP$ in base $b$ we have 
\begin{equation}\label{disc0m2}
D_{b^m}(\cP)\le \frac{1}{b^m}\left(c_b m+9+\frac{4}{b}\right), \ \ \ \mbox{ where }\  c_b:=\left\{
\begin{array}{ll}
\frac{b^2}{b+1} & \mbox{ if $b$ is even},\\
b-1 & \mbox{ if $b$ is odd}. 
\end{array}
\right.
\end{equation}
So the discrepancy of a $(0,m,2)$-net is of order of magnitude $O((\log N)/N)$, where $N=b^m=|\cP|$. According to a celebrated result by Schmidt~\cite{Schm72distrib} this order is the best possible for the discrepancy of any $N$-element plane point set.\\

Our aim is to present two interesting construction algorithms which can in principle construct every $(0,m,2)$-net in base $b$. This way we construct point sets with optimal order of star discrepancy.

\section{The first algorithm}

If we are given a $(0,m,s)$-net in base $b$, say $\cP=\{\bsx_1,\bsx_2,\ldots,\bsx_{b^m}\}$, then each point $\bsx_n$ belongs to a $b$-adic $b^m \times b^m$ box of the form 
\begin{equation}\label{boxm}
\prod_{j=1}^s \left[\frac{u_j(n)}{b^{m}},\frac{u_j(n)+1}{b^{m}}\right)
\end{equation}
where $u_j(n)=\lfloor b^m x_{n,j}\rfloor$ whenever $x_{n,j}$ is the $j$th component of $\bsx_n$. It is elementary to see that the $(0,m,s)$-net property remains valid if one shifts the elements of a $(0,m,s)$-net within their corresponding $b$-adic boxes of the form \eqref{boxm}.

So, instead of constructing point sets, we are now going to construct sets of $b$-adic boxes of the form \eqref{boxm}.

For a given box $X=\prod_{j=1}^s [\tfrac{u_j}{b^{m}},\tfrac{u_j+1}{b^{m}})$ let $\cE_m(X)$ be the set of all $b$-adic boxes of volume $b^{-m}$ which contains $X$ as a subset, $$\cE_m(X)=\{E \ : \ \mbox{$E$ is a $b$-adic box of volume $b^{-m}$ with $X \subseteq E$}\}.$$

We propose the following algorithm for the construction of finite sequences of $b$-adic boxes $X=\prod_{j=1}^2 [\tfrac{u_j}{b^{m}},\tfrac{u_j+1}{b^{m}})$:

\begin{algorithm}[H]
\caption{Construction of a finite sequence of boxes of the form \eqref{boxm}}\label{alg1}
\begin{algorithmic}[1]
\STATE \textbf{Input:} base $b$, resolution $m$
\STATE Set $n=1$;
\STATE Set $\cU_1=[0,1)^2$;
\REPEAT
\STATE Choose an arbitrary box $$X_n=\prod_{j=1}^2 \left[\frac{u_j(n)}{b^{m}},\frac{u_j(n)+1}{b^{m}}\right)\subseteq \cU_n ;$$ 
\STATE Set $n=n+1$;
\STATE Set $$\cU_n=\cU_{n-1} \setminus \bigcup_{E \in \cE_m(X_{n-1})}E;$$
\UNTIL{$\cU_n=\emptyset$}
\RETURN $X_1,X_2,\ldots,X_n$
\end{algorithmic}
\end{algorithm}

An example of a construction according to Algorithm~\ref{alg1} is illustrated in Figure~\ref{f1}.

\begin{figure}[htp]
\begin{center}
\input{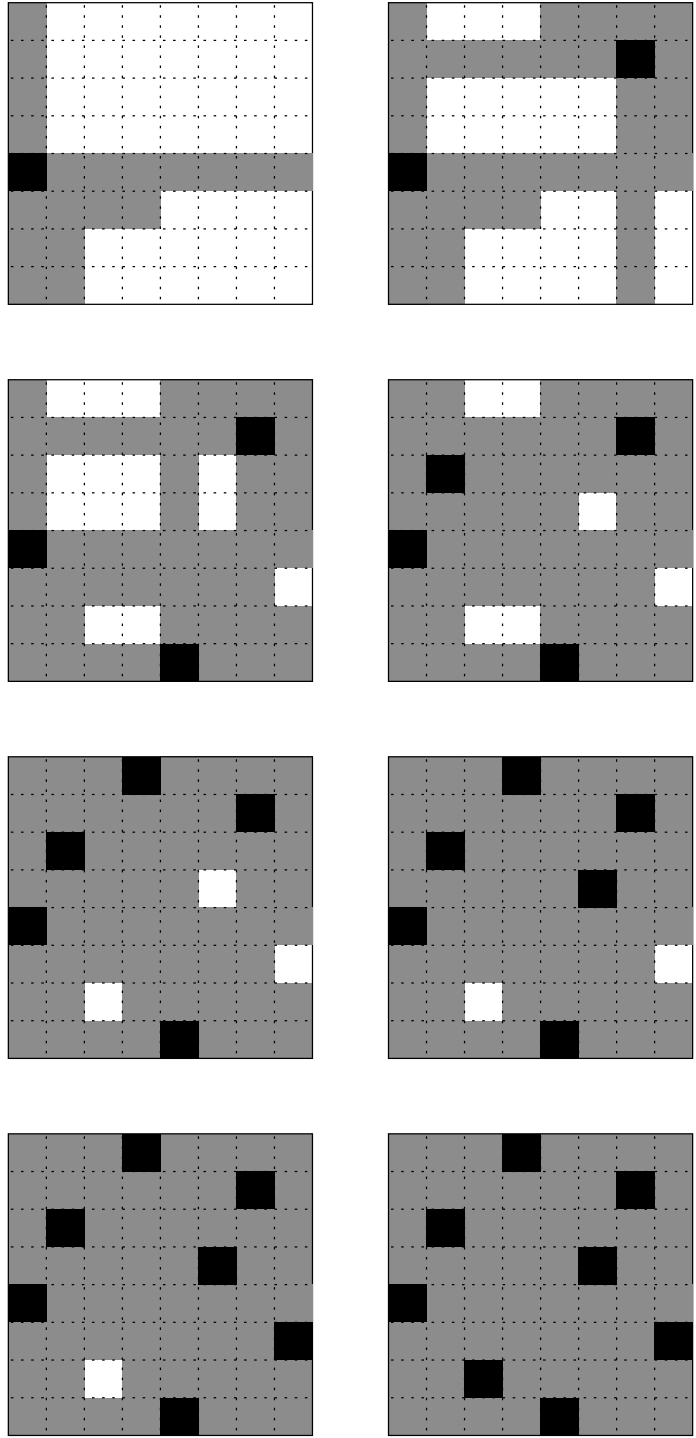}
     \caption{Algorithm~\ref{alg1} for $b=s=2$ and $m=3$; The chosen $X_1,\ldots,X_8$ are colored black. The corresponding $b$-adic boxes are in dark gray. }
     \label{f1}
\end{center}
\end{figure}

\begin{theorem}
Algorithm~\ref{alg1} terminates after exactly $b^m$ steps and the outcome yields a $(0,m,2)$-net in base $b$. In particular, the so constructed point set satisfies the discrepancy bound \eqref{disc0m2}.
\end{theorem}

We split the proof into several short lemmas:

\begin{lemma}\label{pr1}
Algorithm \ref{alg1} terminates after at most $b^{m}$ steps, i.e., $n \le b^m$.
\end{lemma}

\begin{proof}
Each $X_j$ belongs to exactly one $b$-adic interval of the form $$\left[\frac{a}{b^m},\frac{a+1}{b^m}\right) \times [0,1) \ \ \ \ \mbox{ for $a\in \{0,1,\ldots,b^m-1\}$.}$$ Hence the result follows since $$\bigcup_{a=0}^{b^m -1}  \left[\frac{a}{b^m},\frac{a+1}{b^m}\right) \times [0,1) =[0,1)^2.$$
\end{proof}

\begin{lemma}\label{pr2}
If Algorithm~\ref{alg1} terminates after exactly $b^{m}$ steps, then the output sequence $X_1,X_2,\ldots,X_{b^m}$ constitutes a $(0,m,2)$-net in base $b$.
\end{lemma}

For the proof we need the following easy lemma which we state for arbitrary dimension $s$:

\begin{lemma}\label{le1}
\begin{enumerate}
\item There are exactly $b^m {m+s-1 \choose m}$ many $s$-dimensional $b$-adic intervals with volume $b^{-m}$.
\item Every $b$-adic interval of the form $X=\prod_{j=1}^s [\tfrac{u_j}{b^{m}},\tfrac{u_j+1}{b^{m}})$ is contained in exactly ${m+s-1 \choose m}$ $s$-dimensional $b$-adic intervals with volume $b^{-m}$, i.e., $|\cE_m(X)|={m+s-1 \choose m}$.
\end{enumerate} 
\end{lemma}

\begin{proof}
\begin{enumerate}
\item The requested number is given by $$\sum_{d_1,\ldots,d_s=0\atop d_1+\cdots+d_s=m}^{\infty} \sum_{a_1=0}^{b^{d_1}-1}\cdots \sum_{a_s=0}^{b^{d_s}-1} 1=\sum_{d_1,\ldots,d_s=0\atop d_1+\cdots+d_s=m}^{\infty} b^{d_1+\cdots+d_s}=b^m {m+s-1\choose m}.$$
\item For fixed $(d_1,\ldots,d_s) \in \NN_0^s$ with $d_1+\cdots+d_s=m$ there is exactly one choice $$(a_1,\ldots,a_s) \in \prod_{j=0}^s \{0,1,\ldots,b^{d_j}-1\}$$ such that $$\prod_{j=1}^s \left[\frac{u_j}{b^{m}},\frac{u_j+1}{b^{m}}\right) \subseteq \prod_{j=1}^s \left[\frac{a_j}{b^{d_j}},\frac{a_j+1}{b^{d_j}}\right).$$ Hence the requested number is exactly the number of $(d_1,\ldots,d_s) \in \NN_0^s$ with $d_1+\cdots+d_s=m$ and this is ${m+s-1 \choose m}$.
\end{enumerate}
\end{proof}

Now we give the proof of Lemma~\ref{pr2}:

\begin{proof}
Assume that Algorithm~\ref{alg1} constructs the finite sequence $X_1,X_2,\ldots,X_{b^m}$. According to Lemma~\ref{le1} we have $|\cE_m(X_{\ell})|={m+1\choose m}=m+1$. Each $b$-adic interval $E \in \cE_m(X_{\ell})$ contains exactly one element of $X_1,\ldots,X_{b^m}$, namely $X_\ell$, which is the correct portion. Of course, $$\cE_m(X_{\ell}) \cap \cE_m(X_{k})=\emptyset \ \ \ \mbox{ whenever $\ell\not=k$.}$$ Hence the number of $b$-adic boxes of volume $b^{-m}$ which contain exactly one of $X_1,X_2,\ldots,X_{b^m}$ is given by $$\sum_{\ell=1}^{b^m}| \cE_m(X_{\ell})| = b^m (m+1).$$ This is already the number of all $b$-adic boxes in dimension $2$ with volume $b^{-m}$. Hence, every $b$-adic box in dimension $2$ with volume $b^{-m}$ contains exactly one of $X_1,X_2,\ldots,X_{b^m}$.  This means that $X_1,X_2,\ldots,X_{b^m}$ constitute a $(0,m,2)$-net in base $b$.
\end{proof}

\begin{lemma}\label{pr3}
Algorithm~\ref{alg1} terminates after exactly $b^m$ steps. 
\end{lemma}

\begin{proof}
Assume that Algorithm~\ref{alg1} terminates after the step $n < b^m$. Then there exists a $k \in \{0,1,\ldots,b^m-1\}$ such that $$X_{\ell} \not\subseteq \left[\frac{k}{b^m},\frac{k+1}{b^m}\right)\times [0,1)\ \ \ \mbox{ for all $\ell=1,2,\ldots,n$.}$$  
We may assume without loss of generality that $k=0$. Hence $X_{\ell} \not\subseteq [0,\tfrac{1}{b^m})\times [0,1)$ for all $\ell=1,2,\ldots,n$.

Now consider elementary boxes of the form 
\begin{equation}\label{step2}
\left[0,\frac{1}{b^{m-1}}\right) \times \left[\frac{k}{b},\frac{k+1}{b}\right) \ \ \ \ \mbox{ for $k=0,1,\ldots,b-1$.}
\end{equation} 
Assume that all of these $b$ boxes contain one of $X_1,\ldots,X_n$. Then 
\begin{equation}\label{pgprinc}
\bigcup_{h=1}^{b-1} \left(\left[\frac{h}{b^m},\frac{h+1}{b^m}\right)\times [0,1)\right)
\end{equation}
contains $b$ intervals $X_j$. According to the pigeonhole principle there must be an interval among the $b-1$ intervals in union \eqref{pgprinc} which contains two $X_j$'s. This however is impossible due to the definition of the algorithm.
Thus we have shown that at least one of the intervals in \eqref{step2} does not contain a $X_j$.

Again it is no loss of generality if we assume that $X_j \not\subseteq [0,\tfrac{1}{b^{m-1}}) \times [0,\frac{1}{b})$ for all $j=1,\ldots,n$.

Now we consider elementary boxes of the form 
\begin{equation}\label{step3}
\left[0,\frac{1}{b^{m-2}}\right)\times \left[\frac{k}{b^2},\frac{k+1}{b^2}\right) \ \ \ \ \mbox{ for $k=0,1,\ldots,b-1$.}
\end{equation}

Assume that all of these $b$ boxes contain one of $X_1,\ldots,X_n$. Then 
\begin{equation}\label{pgprinc2}
\bigcup_{h=1}^{b-1} \left(\left[\frac{h}{b^{m-1}},\frac{h+1}{b^{m-1}}\right)\times \left[0,\frac{1}{b}\right)\right)
\end{equation}
contains $b$ intervals $X_j$. According to the pigeonhole principle there must be an interval among the $b-1$ intervals in union \eqref{pgprinc2} which contains two $X_j$'s. This however is impossible due to the definition of the algorithm.
Thus we have shown that at least one of the intervals in \eqref{step3} does not contain a $X_j$.

Again it is no loss of generality if we assume that $X_j \not\subseteq [0,\tfrac{1}{b^{m-2}}) \times [0,\frac{1}{b^2})$ for all $j=1,\ldots,n$.

If we continue this process we finally find that $\cU_n$ contains a $b$-adic box with side length $b^{-m}$. For example if in each step $k=0$ then this is the box $[0,\tfrac{1}{b^m})^s$. Therefore there is space to choose a further interval $X_{n+1}$.
\end{proof}

\section{Algorithm~\ref{alg1} for arbitrary dimension} 

In principle we can formulate Algorithm~\ref{alg1} also in arbitrary dimension $s \ge 2$:

\setcounter{algorithm}{0}

\renewcommand{\thealgorithm}{\arabic{algorithm}'}

\begin{algorithm}[H]
\caption{Construction of a finite sequence of boxes of the form \eqref{boxm} in dimension $s$}\label{alg2}
\begin{algorithmic}[1]
\STATE \textbf{Input:} base $b$, number of dimensions $s$, resolution $m$
\STATE Set $n=1$;
\STATE Set $\cU_1=[0,1)^s$;
\REPEAT
\STATE Choose an arbitrary box $$X_n=\prod_{j=1}^s \left[\frac{u_j(n)}{b^{m}},\frac{u_j(n)+1}{b^{m}}\right)\subseteq \cU_n ;$$ 
\STATE Set $n=n+1$;
\STATE Set $$\cU_n=\cU_{n-1} \setminus \bigcup_{E \in \cE_m(X_{n-1})}E;$$
\UNTIL{$\cU_n=\emptyset$}
\RETURN $X_1,X_2,\ldots,X_n$
\end{algorithmic}
\end{algorithm}

\renewcommand{\thealgorithm}{\arabic{algorithm}}

Then Algorithm~\ref{alg2} still terminates after at most $b^m$ steps (same proof as for Lemma~\ref{pr1}) and if it terminates after exactly $b^m$ steps, then the output again constitutes a $(0,m,s)$-net in base $b$ (same proof as for Lemma~\ref{pr2}). The problem is that we cannot guarantee that the algorithm runs until the $b^m$th step, i.e., we do not have a counterpart of Lemma~\ref{pr3} for dimension $s \ge 3$. Even for $s=3$ and $b=2$ there are instances where the algorithm stops before $b^m$. An example is illustrated in Figure~\ref{f2}. A necessary condition in order that Algorithm~\ref{alg2} runs until $n=b^m$ is of course that $s \le b+1$ since otherwise a $(0,m,s)$-net in base $b$ cannot exist.

\begin{figure}[h]
\begin{center}
\input{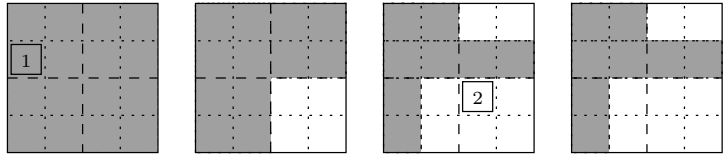} 
\caption{Situation in dimension $s=3$ and base $b=2$ where Algorithm~\ref{alg2} stops already after 2 steps. Each square is a slice of the cube of height $1/4$.}
     \label{f2}
\end{center}
\end{figure}

\section{The second algorithm} The second proposed algorithm is a recursive construction in two dimensions which mimics the construction of Hammersley. It relies on the observation, that every $(0,m,2)$-net in base $b$ with 
$m\ge 2$ induces a $(0,m-1,2)$-net on each of the $b$ rectangles with 
exactly one side length equal to $b^{-1}$, that partition the unit square.

The algorithm we propose synthesizes a $(0,m,2)$-net by (recursively)
generating $b$ (not necessarily different) $(0,m-1,2)$-nets, scaling
and shifting them and adjusting the points in a way that guarantees
that the net-property is satisfied.

Let $\QQ(b^m):=\{0,\tfrac{1}{b^m},\tfrac{2}{b^m},\ldots,\tfrac{b^m-1}{b^m}\}$ be the set of $b$-adic rationals with denominator not larger than $b^m$.
Recall that any $b$-adic $b^{-m} \times b^{-m}$ box is uniquely described by 
the coordinates of its left lower corner. So we assume in the following without loss of generality that the elements of a $(0,m,2)$-net in base $b$ belong to $\QQ(b^m)\times \QQ(b^m)$.

Let ${\cal S}_b$ denote the set of permutations of the $b$ elements
$\{0,\ldots,b-1\}$. 

We also need the following two mappings: 
\begin{itemize}
\item for $b \in \NN$, $b \ge 2$, let $A_b:\RR^2\rightarrow \RR^2$ be the linear scaling by $b^{-1}$ in the direction of the first coordinate, i.e., 
$A_b(x,y)=(x/b,y)$, and
\item for $m \in \NN$, and permutations $\pi_0,\ldots ,\pi_{b^m -1} \in {\cal S}_b$ let $\psi_{b,m}:\QQ(b^{m})\times \QQ(b^{m-1}) \rightarrow \QQ(b^{m})\times \QQ(b^{m})$, $$\psi_{b,m}(x,y):=
(x,y+\pi_{b^{m-1}y}(x_1)b^{-m}),\ \ \ \mbox{ where $x_1=\lfloor b x\rfloor$.}$$
\end{itemize}

\begin{algorithm}[H]
\caption{Recursive construction of a $(0,m,2)$-net in base $b$}\label{alg3}
\begin{algorithmic}[1]
\STATE \textbf{Input:} base $b$, final resolution $m$
\STATE Set $\cP_0=\{(0,0)\}$;
\FOR{$n=1$ to $m$}
\STATE Choose $b^{n-1}$ permutations $\pi_0^{(n)},\ldots,\pi_{b^{n-1}-1}^{(n)}\in {\cal S}_b$; 
\STATE Set $\widehat{\cP}_n:=\bigcup_{j=0}^{b-1} A_b(\cP_{n-1}+(j,0))$;
\STATE Set $\cP_n:=\psi_{b,n}(\widehat{\cP}_n)$
\ENDFOR
\RETURN $\cP_m$
\end{algorithmic}
\end{algorithm}

\begin{theorem}
The output $\cP_m$ of Algorithm~\ref{alg3} is a $(0,m,s)$-net in base $b$, and therefore its discrepancy satisfies the bound \eqref{disc0m2}.
\end{theorem}

The proof of this result follows from the observation that $\cP_0$ is a $(0,0,2)$-net in base $b$ and by the following proposition which is slightly more general than necessary.

\begin{proposition}
Let $b,m\in \NN$, $b \ge 2$. Let $b$ (not necessarily different) $(0,m-1,2)$-nets 
$\cP_{j}$ in base $b$ for $j=0,\ldots,b-1$ be given 
as well as $b^{m-1}$ permutations
$\pi_0,\ldots,\pi_{b^{m-1}-1}\in {\cal S}_b$.  
Construct a set $\cP$ of $b^m$ points as follows: put
$$
\widehat\cP:=\bigcup_{j=0}^{b-1} A_b(\cP_{j}+(j,0)) \ \ \ \mbox{ and } \ \ \ 
\cP:=\psi_{b,m}(\widehat\cP).$$
Then $\cP$ is a $(0,m,2)$-net in base $b$.
\end{proposition}

\begin{proof}
We need to show that every elementary interval  of volume $b^{-m}$ contains
exactly one element of $\cP$.
Consider first an elementary interval $I$ of the form $$I=\left[\frac{a_1}{b^k},\frac{a_1+1}{b^k}\right)\times \left[\frac{a_2}{b^{m-k}},\frac{a_2+1}{b^{m-k}}\right)$$ with 
$a_1\in \{0,\ldots,b^{k}-1\}$, $a_2\in\{0,\ldots,b^{m-k}-1\}$ and with 
$k\in \{1,\ldots,m\}$. We may thus write $$I=A_b(J+(i,0)),$$ where $i=\lfloor a_1b^{-(k-1)}\rfloor$ and where $J$ is the elementary interval $$J=\left[\frac{a_1-i b^{k-1}}{b^{k-1}},\frac{a_1-i b^{k-1}+1}{b^{k-1}}\right)\times \left[\frac{a_2}{b^{m-k}},\frac{a_2+1}{b^{m-k}}\right)$$
of volume $b^{-(m-1)}$. This interval $J$ contains 
precisely one element of $\cP_i$ and hence $I$ contains precisely one element of $A_b(\cP_i+(i,0))$. On the other hand we have $A_b(\cP_j+(j,0)) \cap I=\emptyset$ for $j \not=i$. Thus $I$ contains precisely one element
of $\widehat\cP$. Since the second coordinate of this element belongs to $\QQ(b^{m-1})$, and since $\psi_{b,m}$ adds at most $\tfrac{b-1}{b^m}$ 
to the second coordinate, $I$ contains precisely one element of $\cP$.\\

It remains to consider the case that $I$ is an elementary interval of the form $$I=[0,1)\times \left[\frac{a_2}{b^m},\frac{a_2+1}{b^m}\right).$$ 
Let $j= \lfloor a_2 b^{-1} \rfloor$.
By construction, each of the $b$ elementary intervals 
$$\left[\frac{\ell}{b},\frac{\ell+1}{b}\right)\times \left[\frac{j}{b^{m-1}},\frac{j+1}{b^{m-1}}\right),$$ for $\ell=0,\ldots,b-1$, contains 
precisely one element of $\cP$. But by construction of $\psi_{b,m}$, 
precisely for one of these elements the second coordinate equals 
$a_2 b^{-m}$.
\end{proof}

\begin{remark}\rm
If we choose the identity for every permutation in Algorithm~\ref{alg3}, then we obtain the Hammersley net in base $b$ as output, i.e., $\cP_m=\cH_{m,b}$ (see \eqref{Ham}). 
\end{remark}

\noindent{\bf Author's Addresses:}

\noindent Gunther Leobacher and Friedrich Pillichshammer, Institut f\"{u}r Finanzmathematik und angewandte Zahlentheorie, Johannes Kepler Universit\"{a}t Linz, Altenbergerstra{\ss}e 69, A-4040 Linz, Austria. Email: ralph.kritzinger(at)jku.at, friedrich.pillichshammer(at)jku.at\\

\noindent Thomas Schell, Softwaredeveloper Calculation Apps, CORPORATE Information Services, PALFINGER AG, F.-W.-Scherer-Strasse 24, A-5020 Salzburg, Austria. Email: ts(at)0-1.at

\end{document}